\documentclass[a4paper]{amsart}
\usepackage{graphicx}
\usepackage{amsmath,amssymb,amsthm,lmodern,mathdots,microtype,tikz-cd}
\usepackage[cmtip,all]{xy}
\usepackage[T1]{fontenc}
\usepackage[colorlinks=true,allcolors=blue!60!black]{hyperref}
\usepackage[capitalise]{cleveref}
\usepackage[left=3.5cm,right=3.5cm,top=3.25cm,bottom=3.25cm]{geometry}
\usepackage{setspace}
\usepackage{bbold}
\usepackage{comment}

\let\amsamp=&

\allowdisplaybreaks

\theoremstyle{plain}
\newtheorem{thm}{Theorem}[subsection]
\newtheorem*{thm*}{Theorem}
\newtheorem{prop}[thm]{Proposition} 
\newtheorem*{prop*}{Proposition} 
\newtheorem{lemma}[thm]{Lemma}
\theoremstyle{definition} 
\newtheorem{defn}[thm]{Definition}
\newtheorem*{note*}{Note}
\newtheorem{rem}[thm]{Remark}

\newtheorem{notation}[thm]{Notation}

\newcommand{\id}{\mathrm{id}} 
\newcommand{\kk}{R} 
\newcommand{\N}{\mathbb{N}} 
\newcommand{\Z}{\mathbb{Z}} 
\DeclareMathOperator{\Coker}{Coker} 

\newcommand{\Aa}{\mathcal{A}}

\newcommand{\Ee}{\mathcal{E}}

\newcommand{\Zzw}{\mathcal{ZW}}
\newcommand{\Bbw}{\mathcal{BW}}

\newcommand{\ncpx}{N\text{-}\mathrm{mC}_\kk} 
\newcommand{\cpx}[1]{#1\text{-}\mathrm{mC}_\kk} 
\newcommand{\Mod}{\text{-}\mathrm{Mod}} 
\newcommand{\bimod}{\mathsf{bgMod}_\kk} 		

\newcommand{\vmat}[2]{\mbox{\tiny$\begin{pmatrix} #1\\ #2\end{pmatrix}$}}

\newcommand{\inv}{\mathfrak o} 

\newcommand{\EI}{\mathit{'\hspace{-0.1em}E}}
\newcommand{\EII}{\mathit{''\hspace{-0.1em}E}}
\newcommand{\II}{\mathit{'\hspace{-0.1em}I}}
\newcommand{\JI}{\mathit{'\hspace{-0.1em}J}}
\newcommand{\III}{\mathit{''\hspace{-0.1em}I}}
\newcommand{\JII}{\mathit{''\hspace{-0.1em}J}}

\title[Model category structures on truncated multicomplexes]{Model category structures on truncated multicomplexes for complex geometry}

\author{Joana Cirici}
\address[Joana Cirici]{Departament de Matem\`atiques i Inform\`atica, Universitat de Barcelona, Gran Via 585, 08007
Barcelona, Spain}
\email{jcirici@ub.edu}

\author{Muriel Livernet}
\address[Muriel Livernet]{Universit\'e Paris Cit\'e,  Institut de Mathématiques de Jussieu-Paris Rive Gauche, F-75013 Paris, SU, CNRS, DMA, ENS-PSL}
\email{muriel.livernet@imj-prg.fr}

\author{Sarah Whitehouse}
\address[Sarah Whitehouse]{
School of Mathematical and Physical Sciences\\
University of Sheffield\\ Sheffield\\ S3 7RH\\ UK}
\email{s.whitehouse@sheffield.ac.uk }

\subjclass[2020]{
18N40, 
18G40, 
32Q55
}

\keywords{multicomplex, spectral sequence, model category, complex manifold, pluripotential homotopy theory}

\thanks{J.~C.~acknowledges Govern de Catalunya (2021-SGR-00697 and Serra H\'{u}nter Program), the Spanish State Research Agency (CEX2020-001084-M, EUR2023-143450, PID2020-117971GB-C22 and PID2024-155646NB-I00) and the  Agence Nationale pour la Recherche through project ANR-20-CE40-0016 HighAGT.\\  S.~W.~ and M.~L.~ would like to thank the Isaac Newton Institute for Mathematical Sciences, Cambridge, for support and hospitality during the programme Equivariant homotopy theory in context, where some of the work on this paper was undertaken. This work was supported by EPSRC grant EP/Z000580/1.}

\begin{document}

\begin{abstract} To a bicomplex one can associate two natural filtrations, the column and row filtrations, 
and then two associated spectral sequences. This can be generalized to $N$-multicomplexes. We present a family of model 
category structures on the category of $N$-multicomplexes where the weak equivalences are the morphisms 
inducing a quasi-isomorphism at a fixed page $r$ of the first spectral sequence and at a fixed page $s$ of the second spectral sequence.
Such weak equivalences arise naturally in complex geometry. In particular, the model structures presented here 
establish a basis for studying homotopy types of almost  complex manifolds.

\end{abstract}
\maketitle

\section{Introduction}
The equations satisfied by the differentials of a bicomplex or, more generally, the structure maps of an $N$-multicomplex, 
are symmetric. When doing homological algebra of bicomplexes, one usually breaks the symmetry by considering  
the spectral sequence associated to either the column or the row filtration. Note that there is an involution  on the category 
of bicomplexes which exchanges the roles of the column and row spectral sequences. Similarly, the
category of $N$-multicomplexes (also known as $N$-truncated multicomplexes or twisted complexes) also enjoys an 
involution giving two naturally associated spectral sequences. For each integer $r\geq 0$ there is a class of weak equivalences  
given by those maps inducing a quasi-isomorphism at the $r$-page of the chosen spectral sequence. 
Model structures with respect to such classes of weak equivalences have been previously constructed in \cite{celw} and \cite{MuRo} 
for bicomplexes and in \cite{fglw22} for the case of multicomplexes. 

Readily, the question arises whether one could consider more symmetric homological algebra approaches which treat 
both differentials of a bicomplex on equal footing and, more generally, which are invariant under the involution operation 
on $N$-multicomplexes. Surprisingly (and to our knowledge), such theories have only been recently considered, in the setting of 
complex geometry as we will next explain, and are still at an early stage of development. 

The complexified de Rham algebra of any complex manifold admits a bigrading induced by the almost complex structure and, in the integrable case, the 
exterior differential decomposes into two complex conjugate components, $d=\overline\partial+\partial$, giving a bicomplex.
Cohomology with respect to $\overline{\partial}$ is known as \textit{Dolbeault cohomology} and  cohomology with respect to 
$\partial$ is often referred to as \textit{anti-Dolbeault cohomology}.
Recent work of Angella-Tomassini~\cite{AnTo} shows that the consideration of symmetric Massey-like triple products for 
complex manifolds, treating the two components $\overline{\partial}$ and $\partial$ on equal footing,  give interesting and new 
biholomorphic invariants. For instance, in contrast with ordinary Massey products, such products do not necessarily vanish for compact Kähler manifolds~\cite{SPZ}.
It turns out that, at least in the compact case, the homotopy-theoretic notion of formality that is obstructed by these symmetric Massey 
products is governed by those maps of bicomplexes inducing isomorphisms both in Dolbeault and anti-Dolbeault cohomology~\cite{MiSte}. 
Such symmetric Massey products may be defined for any algebra in the category of bicomplexes, for which Dolbeault and anti-Dolbeault cohomologies are just the first pages
of the column and row spectral sequences respectively.
A model structure on bicomplexes related to this set-up is given in~\cite{Stelzig} as a 
special case of a general result on categories of modules over Frobenius algebras. 

A main motivation for the present note is to give a more transparent and direct construction of a symmetric-type model structure, 
exhibiting the sets of generating (trivial) cofibrations at the same time as we generalize the theory in two different directions: 
first, we study $N$-multicomplexes rather than bicomplexes and, second, we consider the class of weak equivalences given by isomorphisms 
at higher stages of the associated spectral sequences, rather than only the first pages.

Our main result is the following. Let $R$ be a commutative unital ring. 
We consider $N$-multicomplexes of $R$-modules, where $N\geq 2$ is a fixed finite natural number. The total differential decomposes 
as $d=d_0+\cdots+d_{N-1}$ and so the case $N=2$ is just the bicomplex case.
Given integers $r,s\geq 0$, we let $\mathcal{E}_{r,s}$ denote the class of (strict) multicomplex maps inducing quasi-isomorphisms 
at the $\EI_r$ and $\EII_s$ pages respectively, where  $\EI$ 
and $\EII$ denote the two spectral sequences.
The classes of weak equivalences $\mathcal{E}_{r,s}$ 
contain all isomorphisms, satisfy the two-out-of-three property, are closed under retracts
and satisfy $\mathcal{E}_{r,s}\subseteq \mathcal{E}_{r'.s'}$ for all $r\leq r'$ and $s\leq s'$. In particular, the smallest class  
is $\mathcal{E}_{0,0}$, given by those maps of $N$-multicomplexes inducing isomorphisms at the first page of both spectral sequences.
Our main result is Theorem~\ref{T:model} in which we 
establish a 
right proper combinatorial model structure on the category of $N$-multicomplexes, with $\mathcal{E}_{r,s}$
as the class of weak equivalences. Fibrations in this case are given by those maps inducing bidegreewise surjective maps at 
$\EI_i$ and $\EII_j$, for every $0\leq i\leq r$ and 
$0\leq j\leq s$.

Let us next explain the relation of this result with existing work as well as some of its potential applications to geometry.

\subsection*{Previously studied asymmetric model structures}
As mentioned before, there is a model structure on $N$-multicomplexes with the class of weak equivalences $\mathcal{E}_r$ 
given by those (strict) multicomplex morphisms inducing a quasi-isomorphism at $\EI_r$ \cite{fglw22}. In this case, 
fibrations are  given by those maps inducing bidegreewise surjective maps at  $\EI_i$ for every $0\leq i\leq r$.
    We obviously have strict inclusions $\mathcal{E}_{r,s}\subset \mathcal{E}_r$
    for all $s\geq 0$ and, in fact, the identity adjunction gives a Quillen adjunction between the model categories associated to $\mathcal{E}_{r,s}$
    and $\mathcal{E}_r$ respectively. Note that this adjunction is not an equivalence.
    As shown in this note, the corresponding sets of generating (trivial) cofibrations are strongly related: Let $\II_r$ and $\JI_r$ denote the sets of generating cofibrations and trivial cofibrations respectively, given in the model structure with $\mathcal{E}_r$-equivalences of \cite{fglw22}. Define 
    $\III_r:=(\II_r)^\inv$ and $\JII_r:=(\JI_r)^\inv$, where $(-)^{\inv}$ denotes the involution on the category of $N$-multicomplexes. Then the sets of generating cofibrations and trivial cofibrations for the model structure with $\mathcal{E}_{r,s}$-equivalences are just given by $\II_r\cup \III_s$ and $\JI_r\cup \JII_s$ respectively.

\subsection*{Bicomplexes and complex manifolds}
There are two somewhat dual cohomological constructions, called \textit{Bott-Chern} and \textit{Aeppli cohomologies} respectively, that are special to bicomplexes. These were introduced in the setting of complex geometry and shown to interpolate between (anti-)Dolbeault and de Rham cohomologies. The class $\mathcal{W}$ of bicomplex maps inducing isomorphisms on both Bott-Chern and Aeppli cohomologies satisfies 
$\mathcal{W}\subseteq \mathcal{E}_{0,0}$ and, 
when restricted to the category of locally bounded bicomplexes we actually have an equality $\mathcal{W}=\mathcal{E}_{0,0}$. This is true, for instance, in the category of compact complex manifolds.
The model structure on the category of bicomplexes obtained by Stelzig in \cite{Stelzig} has $\mathcal{W}$ as the class of weak equivalences. Our model structure with $\mathcal{E}_{0,0}$ as weak equivalences is a right Bousfield localization of the model structure with $\mathcal{W}$ as weak equivalences.
 As mentioned before, a main motivation for doing homotopical algebra with respect to the class $\mathcal{W}$ is to compute new refined biholomorphic invariants for complex manifolds. The consideration of the weaker classes of equivalences $\mathcal{E}_{r,s}$ may allow for the definition and computation of such invariants.

 \subsection*{$4$-multicomplexes and almost complex manifolds}
 The above ideas apply to $N$-multicomplexes as well. A natural geometric source for multicomplexes is that of almost complex geometry. Indeed, the complexified de Rham algebra of an almost complex manifold still decomposes into bidegrees, but in this case the 
exterior differential decomposes into  
    four components,    $d=\overline{\mu}+\overline{\partial}+\partial+\mu$. 
    The component $\mu$ and its complex conjugate $\overline\mu$ are manifestations of the Nijenhuis tensor, in the sense that integrability happens if and only if $\overline{\mu}=\mu=0$ (and in this case one has a complex manifold).
A Dolbeault cohomology theory and a Frölicher-type spectral sequence for almost complex manifolds was introduced in \cite{CWDol} and further studied by various authors (see for instance \cite{CPS}, \cite{SiTo}, \cite{CWfour}).
 Note that a standard definition for Bott-Chern or Aeppli cohomologies for the case of $N$-multicomplexes is not avaliable in the literature, although there are some candidates in the case $N=4$ (see for instance \cite{CPS}). The current state-of-the-art asks for a study of almost complex manifolds through the lenses of $\mathcal{E}_{r,s}$-equivalences considered in the present note. In this case, Dolbeault cohomology in the sense of \cite{CWDol} corresponds to the $\EI_2$-page, and so it might be particularly interesting to study the homotopy types of almost complex manifolds with respect to the class $\mathcal{E}_{2,2}$. 
 This is  uncharted territory and the present note paves the way for first exploration steps. 
A vast generalization of this set-up, also giving rise to 4-multicomplexes, is that of  \textit{almost Dirac structures} \cite{dirac}. These encompass  
multiple types of geometric settings including symplectic, Poisson, (almost) complex and
 generalized (almost) complex geometry.

  \medskip

\subsection*{Acknowledgements} 
We would like to thank Jonas Stelzig for useful exchanges.

\section{Notations and preliminaries}
\label{S:preliminaries}

\subsection{Multicomplexes and involution}

Throughout this paper, $\kk$ will denote a commutative unital ground ring and $\kk\Mod$ will denote the category of $\kk$-modules.

\begin{defn} A \emph{bigraded $\kk$-module} is a collection of $\kk$-modules $A = \{A^{p,q}\}_{(p,q)\in\Z^2}$. We denote by $\bimod$ the category whose objects are bigraded $\kk$-modules and whose morphisms are maps $f:A\to B$ of bigraded $\kk$-modules of bidegree $(0,0)$, that is, a collection of $\kk$-module morphisms $f^{p,q}:A^{p,q}\to B^{p,q}$.\\ 
For $N\in \N$,  an \emph{$N$-multicomplex}  $A$ is a bigraded $\kk$-module  endowed with a family of maps $\{d_i \colon A \to A\}_{i \ge 0}$ of bidegree $(-i,1-i)$  satisfying for all $l \ge 0$,
  \begin{equation}\label{E:multicomplex}
    \sum_{i+j=l} (-1)^i d_id_j=0,
  \end{equation}
  and for all $i\geq N,\; d_i=0$. 

 A  \emph{\textup(strict\textup) morphism} of $N$-multicomplexes is a map $f$ in $\bimod$ satisfying $d_if=fd_i$ for all $i\geq 0$.
  We denote by $\ncpx$ the category of $N$-multicomplexes and strict morphisms.
\end{defn}

For example, 
a $2$-multicomplex is a bicomplex, with the convention $d_0d_1=d_1d_0$.

We note that, in contrast to the situation in~\cite{fglw22}, for the results of this paper it is important that $N$ is finite and we will not consider the case $N=\infty$. 

In~\cite[Section 3.1]{celw18} it is proved that the category of twisted complexes (where objects are multicomplexes and morphisms are $\infty$-morphisms between them) is closed symmetric monoidal, and we can deduce that this is also the case for the category $\ncpx$. This also holds for the category $\bimod$. The categories $\bimod$ and $\ncpx$ are complete and cocomplete; for $\ncpx$ this is a
consequence of Lemma~\ref{L:basis0} below showing that it can be viewed as a 
module category.

Note as well that our conventions do not agree with the bidegrees and signs 
arising in (almost) complex geometry, but this is of course harmless and the reader can adapt our results to such conventions. Instead, we have chosen our bidegree and sign conventions to agree with the previous works \cite{celw} and \cite{fglw22}.

\begin{notation}\label{algebrasnotation} 
Let us consider $\kk\langle \delta_0,\delta_1,\ldots,\delta_{N-1}\rangle$, the free associative $\kk$-algebra in the monoidal category $\bimod$ generated by symbols $\delta_i$  of bidegree $(-i,1-i)$, for $0\leq i\leq N-1$. Define 
$S_l=  \sum_{i+j=l} (-1)^i \delta_i\delta_j$, for $0\leq l\leq 2(N-1)$. We denote by $\Aa_N$ the $\kk$-algebra obtained as the quotient of $\kk\langle \delta_0,\delta_1,\ldots,\delta_{N-1}\rangle$ by the ideal generated by $\{S_l\, |\, 0\leq l\leq 2(N-1)\}$.
\end{notation}

\begin{rem} The category of $N$-multicomplexes is isomorphic to the category of $\Aa_N$-modules, as shown in the next lemma.
Note that  the definition of $N$-multicomplexes is equivalent to that where one requires $\sum_{i+j=N} d_id_j=0$.  
Hence the category of $N$-multicomplexes is also isomorphic to that of $\Aa'_N$-modules where $\Aa'_N$ is defined similarly 
to $\Aa_N$ replacing $S_l$ by $S'_l=\sum_{i+j=l} \delta_i\delta_j$.
In~\cite{Algebra}
it is shown that, when $R$ is a field of characteristic zero, the algebras $\Aa'_2$ (resp.~$\Aa'_4$) are the universal 
enveloping algebras of Lie algebras $\mathfrak{g}_2$ (resp.~$\mathfrak{g}_4$). These two Lie algebras
 are quasi-isomorphic when given the natural differential $ad_d=[d,-]$. Note that $\mathfrak{g}_2$
is finite-dimensional and has $ad_d=0$, while $\mathfrak{g}_4$ is infinite-dimensional but has 
non-trivial $ad_d$, giving $H^*(\mathfrak{g}_4, ad_d)\cong \mathfrak{g}_2$. 
It is an easy check to see that for any $N$, $\Aa'_N$ is the universal enveloping algebra of a Lie algebra, 
say $\mathfrak{g}_N$. Whether $\mathfrak{g}_N$ is quasi-isomorphic to $\mathfrak{g}_2$ 
when given the natural differential $ad_d$ would be an interesting question, but is not the purpose of our paper.
\end{rem}

\begin{lemma}\label{L:basis0} The category $\ncpx$ is isomorphic to the category of left $\Aa_N$-modules.\\
For each $(p,q)\in\Z^2$, the $\kk$-module $\Aa_N^{p,q}$ is a free $\kk$-module of finite rank. A basis of $\Aa_N$ as a free bigraded $\kk$-module is given by the set
\[\left\{\delta_{N-1}^\epsilon \delta_{i_1}\ldots\delta_{i_l}\delta_0^w\;|\; l\geq 0, \epsilon,w\in\{0,1\}, 0<i_j<N-1\right\}.\]
\end{lemma}

\begin{proof} For $A$ an $N$-multicomplex and $a$ in $A$, the action of $\Aa_N$ is defined by $\delta_i\cdot a=d_i(a)$. 
This correspondence upgrades to an isomorphism of categories between $\ncpx$ and  left $\Aa_N$-modules.
Note that $\Aa_N$ is a quadratic algebra. We use a Gr\"obner basis to prove the second part of the Lemma 
(see e.g. \cite[Section 2]{Grobner95}). The set of words in the ordered alphabet 
\[X=\{\delta_0>\ldots>\delta_{N-1}\}\] 
is well ordered by the degree lexicographic order. 
The leading term of $S_l=\sum_{i+j=l} (-1)^i \delta_i\delta_j$ is $\delta_0\delta_l$ for 
$0\leq l\leq N-1$ and $\delta_{l-N+1}\delta_{N-1}$ 
for $N-1<l\leq 2(N-1)$.  
Using the composition lemma (also called Bergman's diamond lemma) we get that the set 
\[\{S_l\,|\,0\leq l\leq 2(N-1)\}\] 
is a Gr\"obner basis, and that a basis of $\Aa_N$ is given by the set of words in $X$ not having a subword of the form $\delta_0\delta_k, 0\leq k\leq N-1$ or of the form $\delta_k\delta_{N-1}, 1\leq k\leq N-1$. 
\end{proof}

For example 
\[\Aa_2=\kk<\delta_0,\delta_1>/(\delta_0^2,\delta_0\delta_1-\delta_1\delta_0,\delta_1^2)=\kk 1\oplus\kk\delta_0\oplus\kk\delta_1\oplus\kk\delta_1\delta_0.\]

\begin{notation} Let $\inv:\Z^2\to\Z^2$ be the $\Z$-linear map defined by 
\[\inv(p,q)=((N-2)p-(N-1)q,(N-3)p-(N-2)q).\]
We have $\inv^2=\id_{\Z^2}$.
\end{notation}

\begin{lemma}The involution $\inv$ on $\Z^2$ induces an involution, and thus an equivalence of categories
\[(-)^{\inv}\colon \ncpx\to\ncpx\]
defined on objects by $(A^\inv)^{p,q}=A^{\inv(p,q)}$ and $(d^\inv)_i=d_{N-1-i}$ and on morphisms by 
$(f^\inv)^{p,q}=f^{\inv(p,q)}$ for $f:A\to B$.
\end{lemma}

\begin{proof} It is clear that the involution $\inv$ induces an involution on $\bimod$. Let $A\in\ncpx$.
The map $d_i$ is of bidegree $(-i,1-i)$, meaning that its component $d_i^{p,q}$ is a morphism of $\kk$-modules from $A^{p,q}$ to $A^{(p-i,q+1-i)}$. In particular, since the involution $\inv$ is $\Z$-linear, the induced morphism $(d_i)^{\inv}$ has bidegree 
$\inv(-i,1-i)=(-(N-1-i),1-(N-1-i))$.
\end{proof}

Note that if $N=2$ then $\inv(p,q)=(-q,-p)$ and $(d_0)^\inv=d_1, (d_1)^\inv=d_0$.
\smallskip

\subsection{Recollection from previous work}
We collect the results of \cite{fglw22} that are needed for the proof of the main theorem. We aim to be as brief and self-contained as possible.
In \cite{celw18} we have defined a functor $E$ from multicomplexes to spectral sequences. This functor $E$ associates to a multicomplex $A$ a filtered complex and then  a spectral sequence $(E_r(A))_{r\geq 0}$. 
In \cite{fglw22} a family of model category structures on multicomplexes and $N$-multicomplexes is defined,  using the functor $E$. In order to build these model category structures, the authors give another description of the pages of the associated spectral sequence that we recall here.

\begin{defn}(cf. \cite[Definition 3.4]{fglw22}).
  Let $A$ be an $N$-multicomplex and $r \ge 0$.

  Define the \emph{witness $r$-cycles} $ZW_r^{p,q}(A)$ to be the bigraded $\kk$-modules 
  $ZW_0^{p,q}(A) = A^{p,q}$
  and for $r\ge 1$,
  \begin{align*}
    ZW_r^{p,q}(A)=\{(a_0, a_1,\ldots, a_{r-1}) \mid &~a_i \in A^{p-i, q-i}~\text{for}~0\leq i\leq r-1~\text{such~that}\\
      &~\sum_{i+j=l} (-1)^id_ia_j = 0~\text{for}~0\leq l\leq r-1
    \}.
  \end{align*}
 Define the \emph{witness $r$-boundaries} to be the bigraded $\kk$-modules $BW_0^{p,q}(A) = 0$, $BW_1^{p,q}(A) = A^{p,q}$ and for $r \ge 2$,
  \[
    BW_r^{p,q-1}(A)=ZW_{r-1}^{p+r-1,q+r-2}(A) \oplus A^{p,q-1} \oplus ZW_{r-1}^{p-1,q-1}(A).
  \]

  The bidegree $(0,1)$ map of bigraded $\kk$-modules
  \[
    w_r \colon BW_r^{p,q-1}(A) \to ZW_r^{p,q}(A)
  \] 
  is given by $w_0=0$, $w_1=d_0$ and for $r\ge 2$,
  \begin{multline*}
    (b_0,\dots,b_{r-2}; a; c_0,\dots,c_{r-2})\\
    \overset{w_r}\longmapsto \big(d_0a+\sum_{i=1}^{r-1} (-1)^i d_ib_{r-1-i}, d_1a-\sum_{i=2}^{r} (-1)^i d_ib_{r-i}+c_0, d_2a+ \sum_{i=3}^{r+1} (-1)^i d_ib_{r+1-i}+c_1, \dots, \\ 
    d_{r-1}a+(-1)^{r-1}\sum_{i=r}^{2r-2} (-1)^id_ib_{2r-2-i} +c_{r-2} \big).
  \end{multline*}
\end{defn}

\begin{prop}[cf.~\protect{\cite[Proposition 3.7]{fglw22}}] Let $A$ be an $N$-multicomplex.
  For every $r \ge 0$,  
  \[E_r^{p,q}(A)\cong ZW_r^{p,q}(A)/w_r(BW^{p,q-1}_r(A)).\]
\end{prop}

\begin{rem}Recall that the first page of the spectral sequence is computed using the differential $d_0$. The map induced by $d_1$ on the first page is a differential and the second page of the spectral sequence is the homology with respect to this differential. On higher pages, the differential is described using the witnesses as shown in~\cite{lwz}.
\end{rem}

The witness $r$-cycles and $r$-boundaries define functors $\ncpx\to \kk\Mod$ and they are represented by $N$-multicomplexes $\Zzw^N_r(p,q)$ and $\Bbw_r^N(p,q)$. The map $w_r$ is represented by a morphism of $N$-multicomplexes
$\iota_r \colon \Zzw^N_r(p,q) \to \Bbw^N_r(p,q-1)$  (see Section 3 in \cite{fglw22}).

\begin{defn}\label{defIrJr}   For $r\geq 0$, consider the sets of morphisms of $N$-multicomplexes
  \[I^N_r=\left\{\xymatrix{\Zzw^N_{r+1}(p,q)\ar[r]^-{\iota_{r+1}} & \Bbw^N_{r+1}(p,q-1)}\right\}_{\substack{p,q\in\Z}} 
    \text{ and } 
  J^N_r=\left\{\xymatrix{0\ar[r]& \Zzw^N_{r}(p,q)}\right\}_{\substack{p,q\in\Z}}.\]
  Let $\II_r$ and $\JI_r$ be the sets of morphisms in $\ncpx$ given by 
  \[\II_r:=\cup_{k=1}^{r-1} J^N_k\cup I^N_r\text{ and }
  \JI_r:=\cup_{k=0}^r J^N_k.\]
\end{defn}

\begin{thm}(cf.~\cite[Theorem 3.30]{fglw22}).\label{modelE} 
  For every $r\geq 0$, 
  the category $\ncpx$ admits a right proper cofibrantly generated model structure, denoted $(\ncpx)_r$, where: 
  \begin{enumerate}
    \item A weak equivalence $f$ is a morphism in $\ncpx$ such that $E_r(f)$ is a quasi-isomorphism.
    \item A fibration $f$ is a  morphism in $\ncpx$
      such that $E_i(f)$ is  bidegreewise surjective for every $0\leq i\leq r$. 
    \item $\II_r$ and $\JI_r$ are the sets of generating cofibrations and generating trivial cofibrations respectively. \qed
  \end{enumerate}
\end{thm}

\begin{notation} Let $\III_s=(\II_s)^{\inv}$ and $\JII_s=(\JI_s)^{\inv}$.
\end{notation}

\section{The model category structures}

\subsection{Main Theorem}

Similarly to the situation for bicomplexes, the involution $\inv$ is used to consider another spectral sequence associated to an $N$-multicomplex $A$. Namely, we denote by $\EI(A)$ the spectral sequence $E(A)$  and by $\EII(A)$ the spectral sequence $E(A^\inv)$.

\begin{defn}
 A morphism of $N$-multicomplexes $f \colon A\to B$ is said to be an \textit{$E_{r,s}$-quasi-isomorphism} if
  the morphism $\EI_r(f) \colon \EI_r(A)\to \EI_r(B)$ at the $r$-page of the associated first spectral 
  sequence is a quasi-isomorphism and  the morphism $\EII_s(f) \colon \EII_s(A)\to \EII_s(B)$ at the $s$-page of the associated second spectral 
  sequence is a quasi-isomorphism.
\end{defn}

Denote by $\Ee_{r,s}$ the class of $E_{r,s}$-quasi-isomorphisms. 
This class contains all isomorphisms of $\cpx{N}$, satisfies the two-out-of-three property and is closed under retracts.

\begin{thm}\label{T:model} 
  For every $r,s\geq 0$, 
  the category $\ncpx$ admits a cofibrantly generated model structure, $(\ncpx)_{r,s}$, where: 
  \begin{enumerate}
    \item $\Ee_{r,s}$ is the class of weak equivalences.
        \item A fibration $f$ is a  morphism 
      such that $\EI_i(f)$ and $\EII_j(f)$ are  bidegreewise surjective for every $0\leq i\leq r$ and $0\leq j\leq s$. 
    \item $I_{r,s}=\II_r\cup \III_s$ and $J_{r,s}=\JI_r\cup \JII_s$ are the sets of generating cofibrations and generating trivial cofibrations respectively. 
     \end{enumerate}
Moreover, the model category $(\ncpx)_{r,s}$ is right proper, combinatorial and tractable.
\end{thm}

We prove this theorem in subsection \ref{S:proof}. Before proving it we need some technical lemmas on the computation of the two spectral sequences  $\EI(\Zzw^N_r(p,q))$ and $\EII(\Zzw^N_s(p,q))$.

\subsection{Spectral sequences of the witness cycles}
We start by recalling the construction of the $N$-multicomplexes $\Zzw^N_r(p,q)$ representing the witness $r$-cycles.

\begin{defn}(cf.~\cite[Definition 3.15]{fglw22})\label{D:curlyZW}
The $N$-multicomplex $\Zzw^N_0(p,q)$ is the free left $\Aa_N$-module generated by an  element $x_{p,q}$ of bidegree $(p,q)$.
  Define the $N$-multicomplex $\Zzw^N_1(p,q)$ to be the pushout
  \begin{center}
    \begin{tikzcd}
      \Zzw^N_0(p,q+1) \arrow[r,"d_0^*"] \arrow[d] & \Zzw^N_0(p,q) \arrow[d,"j_0"] \\
      0 \arrow[r] & \Zzw^N_{1}(p,q)
    \end{tikzcd}
  \end{center}
  in the category of $N$-multicomplexes, and for $r \geq 2$, define $\Zzw^N_r(p,q)$ recursively to be the pushout 
  \begin{center}
    \begin{tikzcd}
      \Zzw^N_0(p-r+1,q-r+2) \arrow[r, "d_0^*"] \arrow[d,"D_{r-1}^*"'] & \Zzw^N_0(p-r+1,q-r+1) \arrow[d,"j_{r-1}"] \\
      \Zzw^N_{r-1}(p,q) \arrow[r,"i_{r-1}"] & \Zzw^N_{r}(p,q)
    \end{tikzcd}
  \end{center}
  in the category of $N$-multicomplexes. 
The morphism $d_0^*$ is given by
  \[
      d_0^*(x_{p-r+1,q-r+2}) = d_0x_{p-r+1,q-r+1}.
  \]
  For $r\geq 1$ we  denote the element $j_{r-1}(x_{p-r+1,q-r+1})$ by $a_{r-1}$.
 For $r \geq 2$, we recursively define the morphism $D_{r-1}^*$ to be
  \[
    D_{r-1}^*(x_{p-r+1,q-r+2}) = \sum_{i=1}^{r-1}(-1)^{i+1}d_ia_{r-1-i}.
  \]
By abuse of notation, we denote the elements $i_{r-1}(a_j)$ ($0 \leq j \leq r-2$) by $a_j$.\end{defn}

\begin{rem} Because the categories of $N$-multicomplexes and left $\Aa_N$-modules are isomorphic, the pushout diagrams considered in the definition can be considered as pushout diagrams in the category of left $\Aa_N$-modules. This approach will be more convenient to describe basis elements as was done in Lemma \ref{L:basis0}.
\end{rem}

\begin{prop}(cf.~\cite[Corollary 6.4]{fglw22})\label{P:E1Z} For every $p,q \in \Z$, $k\geq 0$, we have
\[\EI_i(\Zzw^N_{k}(p,q))=\begin{cases}
\kk^{p,q}\oplus\kk^{p-k,q-k+1}& \text{if } 1\leq i\leq k\\
0& \text{if } i\geq k+1.\end{cases}\]
\end{prop}

\begin{prop}\label{P:E2Z} The $N$-multicomplex $\Zzw_k^N(p,q)$ is a free bigraded $\kk$-module. For $k>0$,
 a basis of $\Zzw_k^N(p,q)$ as a free bigraded $\kk$-module is given by the set
\[\left\{\delta_{N-1}^\epsilon \delta_{i_1}\ldots\delta_{i_l}\cdot a_j\;|\; l\geq 0, \epsilon\in\{0,1\}, 0<i_t<N-1, 0\leq j\leq k-1\right\}\]
where the element $a_j$ has bidegree $(p-j,q-j)$.\\
For every $p,q \in \Z$, $k\geq 0$ and $i\geq 1$, we have
\[\EII_i(\Zzw^N_{k}(p,q))=0.\]

\begin{proof} The first assertion is proved by induction on $k\geq 1$. We prove, in addition that in the left $\Aa_N$-module 
$\Zzw^N_k(p,q)$ we have 
 \begin{align}\label{E}
 \delta_0\cdot\left(\sum_{j=1}^{k}(-1)^{j} \delta_j \cdot a_{k-j}\right)=0.
 \end{align}

By definition $\Zzw^N_0(p,q)$ is the free  left $\Aa_N$-module generated by an element denoted $x_{p,q}$ of bidegree $(p,q)$. The map  $d_0^*$ acts on the basis of the free bigraded $\kk$-module $\Zzw^N_0(p,q+1)$ described in Lemma \ref{L:basis0} as
\[d_0^*(\delta_{N-1}^\epsilon\delta_{i_1}\ldots\delta_{i_l}\delta_0^w\cdot x_{p,q+1})=\begin{cases}
\delta_{N-1}^\epsilon\delta_{i_1}\ldots\delta_{i_l}\delta_0 \cdot x_{p,q}&\text{ if } w=0\\
0 &\text{ if } w=1.\end{cases}\]
Denoting by $a_0$ the class of $x_{p,q}$ in $\Zzw^N_1(p,q)=\Coker d_0^*$ we have the result for $k=1$.
In addition, $\delta_0\cdot\delta_1\cdot a_0=\delta_1\cdot (\delta_0 \cdot a_0)=\overline{d_0^*(\delta_1\cdot x_{p,q+1})}=0$. This proves Equation (\ref{E}) for $k=1$. \\
Assume the result is proved for $1\leq j<k$.
We have
\[D_{k-1}^*(\delta_{N-1}^\epsilon\delta_{i_1}\ldots\delta_{i_l}\cdot x_{p-k+1,q-k+2})=
\sum_{j=1}^{k-1}(-1)^{j+1}\delta_{N-1}^\epsilon\delta_{i_1}\ldots\delta_{i_l}\delta_j \cdot a_{k-1-j}\]
where $a_{i}$ has bidegree $(p-i,q-i)$. Moreover, Equation (\ref{E}) implies 
\[D_{k-1}^*(\delta_{N-1}^\epsilon\delta_{i_1}\ldots\delta_{i_l}\delta_0\cdot x_{p-k+1,q-k+2})=0.\]
Therefore, any element of the form
$\delta_{N-1}^\epsilon\delta_{i_1}\ldots\delta_{i_l}\delta_0 \cdot x_{p-k+1,q-k+1}$ in $\Zzw_0(p-k+1,q-k+1)$ is identified with the element $D_{k-1}^*(\delta_{N-1}^\epsilon\delta_{i_1}\ldots\delta_{i_l}\cdot x_{p-k+1,q-k+2})$ and this is the only possible type of identification in the pushout diagram. Hence the bigraded $\kk$-module  splits as $\Zzw_k^N(p,q)=\Zzw_{k-1}^N(p,q)\oplus U(p,q)$ where $U(p,q)$ is the bigraded $\kk$-module spanned by the set $\{\delta_{N-1}^\epsilon \delta_{i_1}\ldots\delta_{i_l}\cdot a_{k-1}\}$.
Let us prove Equation (\ref{E}). By definition we have, for all $u<k$
\[\delta_0 a_u=\sum_{j=1}^{u} (-1)^{j+1}\delta_j\cdot a_{u-j}.\]
As a consequence
\begin{align*}
 \delta_0\cdot\left(\sum_{j=1}^{k}(-1)^{j} \delta_j \cdot a_{k-j}\right)&=
 \sum_{j\geq 1} (-1)^j\left(\sum_{a+b=j,a,b>0} (-1)^{a+1}\delta_a\delta_b\cdot a_{k-j}+(-1)^{j+1}\delta_j\delta_0\cdot a_{k-j}  \right)\\
&= \sum_{a,b>0} (-1)^{b+1}\delta_a\delta_b\cdot a_{k-a-b}+\sum_{j,r>0} (-1)^{r}\delta_j\delta_r\cdot a_{k-j-r}  =0.
\end{align*}
The description of the basis shows that the homology of the complex $\Zzw_k^N(p,q)$ with respect to the differential $d_{N-1}$ is $0$. This homology coincides 
(up to regrading) with the first page of the second spectral sequence, since $(d_{N-1})^\inv=d_0:\Zzw_k^N(p,q)^\inv\to \Zzw_k^N(p,q)^\inv$.
\end{proof}

\end{prop}

\subsection{Proof of Theorem \ref{T:model}}\label{S:proof} The proof is standard and follows the lines of \cite[Theorem 2.1.19]{Hovey}. The class of maps $\Ee_{r,s}$ satisfies the two-out-of-three property and is closed under
retracts.  The domains of $I_{r,s}$ are compact relative to $I_{r,s}$-cell.  The domains of $J_{r,s}$ are compact relative to $J_{r,s}$-cell.
A map $f$ has the right lifting property with respect to $I_{r,s}$ if and only if is has the right lifiting  property with respect to $\II_r$ and $\III_s$. Theorem \ref{modelE} states that this amounts to $f$ satisfying the following conditions: $\EI_r(f)$
is a quasi-isomorphism,
$\EI_i(f)$ is surjective for all $0\leq i\leq r$, $\EI_s(f^\inv)$ is a quasi-isomorphism and $\EI_j(f^\inv)$ is surjective for all $0\leq j\leq s$. In conclusion $f$  has the right lifting property with respect to $I_{r,s}$ if and only if $f$ is an acyclic fibration.  Let $f\colon A\to B$ be a morphism in $\ncpx$ such that $f$ has the left lifting property with respect to fibrations. To show that $f\in\Ee_{r,s}$ we consider the following diagram
\[\xymatrix{ A\ar[r]^-{\vmat{1}{0}}\ar[d]_f & A\oplus B'\ar[d]^{(f,p)}\\ B\ar[r]^{=} & B}\]
where $B'=\Zzw^N_r(0,0)\otimes B\oplus \Zzw^N_s(0,0)^\inv\otimes B$ and $p$ is the projection onto $B$ induced by the projections $q_r\colon \Zzw^N_r(0,0)\to\kk(0,0)$ and $(q_s)^\inv$, where $\kk(0,0)=\kk(0,0)^\inv$ is the bigraded $\kk$-module, free of rank 1 and concentrated in bidegree $(0,0)$  with 
$d_i=0$ for all $i$.
By Proposition \ref{P:E1Z},  the projection $q_r\colon \Zzw^N_r(0,0)\to\kk(0,0)$ satisfies $\EI_i(q_r)$ is surjective for all $0\leq i\leq r$.
This induces that $\EII_j(q_s^\inv)$ is surjective for all $0\leq j\leq s$. 
In particular, we have
$\EI_i(p)$ and $\EII_j(p)$  are surjective for $0\leq i\leq r$ and $0\leq j\leq s$, hence $p$ is a fibration and $f$ admits a lift in the diagram $(g,s)\colon B\to A\oplus B'$. By Propositions \ref{P:E1Z} and \ref{P:E2Z} we have that $\EI_i(\Zzw^N_r(0,0))$ and $\EII_i(\Zzw^N_s(0,0))$ are free bigraded $\kk$-modules and we have $\EI_{r+1}(B')=0=\EII_{s+1}(B')$. 
This implies that both $\EI_r(f)$ and $\EII_s(f)$ are quasi-isomorphisms. \\
The model category is right proper because every object is fibrant. By definition, a combinatorial model category is a cofibrantly generated model category which is locally presentable as a category. We have seen in Lemma~\ref{L:basis0} that the category $\ncpx$ is isormophic to the category of left $\Aa_N$-modules. Hence, it is locally (finitely) presentable by \cite[Example 5.2.2a]{Borceux2}. By ~\cite[Corollary 2.7]{Barwick2010}, to prove that the combinatorial model category is a tractable category it is enough to prove that the morphisms in the set $I_{r,s}$ have cofibrant domain. The $N$-multicomplex $\Zzw_{r+1}^N(p,q)$ is cofibrant in $(\ncpx)_{r+1,s}$ hence is cofibrant in $(\ncpx)_{r,s}$. Similarly $(\Zzw_{s+1}^N(p,q))^\inv$ is cofibrant in $(\ncpx)_{r,s+1}$, hence is cofibrant in $(\ncpx)_{r,s}$.
\qed

\subsection{Comparison with previous results}

The following result compares the model structures of our main Theorem~\ref{T:model}
with the asymmetric model structures of~\cite{fglw22}.

\begin{prop}
Let $N\geq 2$ and $r,s \geq 0$.
    The identity adjunction
 \[\xymatrix{
 (\cpx{N})_r\ar@<1ex>[rr]^-{\id} && (\cpx{N})_{r,s} \ar@<1ex>[ll]^-{\id}_-{\perp} }\]
    is a Quillen adjunction but not a Quillen equivalence.
\end{prop}

\begin{proof}
    It is clear that the right adjoint $(\cpx{N})_{r,s}\to (\cpx{N})_{r}$
    preserves weak equivalences and fibrations, so this is a Quillen
    adjunction.

    Let $p,q \in \Z$, let $\Zzw_\infty^N(p,q)=\varinjlim_s \Zzw_s^N(p,q)$
    and consider $\pi\colon\Zzw_\infty^N(p,q)\to \kk(p,q)$ given by projection to $\kk(p,q)$. 
    By~\cite[Proposition 6.7]{fglw22}, this is a cofibrant replacement of $\kk(p,q)$ in the model category $(\ncpx)_r$
for all $r\geq 0$. But $\EII_1(\Zzw_\infty^N(p,q))=0$, so $\pi$ is not
in $\Ee_{r,s}$ for any $s$.
Thus the left adjoint does not preserve weak equivalences with cofibrant
domain (and fibrant codomain). So the adjunction is not a Quillen equivalence.
\end{proof}

\begin{rem}
    In the case $r=s=0$, $\id\colon (\cpx{N})_{0,0}\to (\cpx{N})_{0}$ is a right Bousfield localization, since the two categories have the same fibrations, namely maps which are surjective, and $\Ee_{0,0}\subset \Ee_0$.  
\end{rem}

\begin{rem}
    The involution gives a Quillen equivalence
    \[(-)^{\inv}\colon (\ncpx)_{r,s}\to (\ncpx)_{s,r}.\]
\end{rem}

\section{Comparing bicomplexes and $4$-multicomplexes}

In this section we would like to compare, with a view towards applications in complex geometry, the homotopy category of bicomplexes and that of $4$-multicomplexes via a Quillen adjunction

\[
    \begin{tikzcd}
      (\cpx{2})_{0,0} \arrow[r, "j"', "\perp" , shift right=1.5] &(\cpx{4})_{1,1} \arrow[l, "q"',shift right=1.5] 
    \end{tikzcd}
    \tag{$\star$}
  \]  
  which is stable under the involution.

\subsection{Definition of the adjunction}
Let $(M,d_0^M,d_1^M)$ be a bicomplex. We define  the 4-multicomplex $(j(M),0,d'_1,d'_2,0)$ as 
  \[ j(M)^{p,q}=M^{q,2q-p},\ d'_1=d_0^M, d'_2=d_1^M.\]
It is clear that $j$ induces a functor $j:\cpx{2}\to\cpx{4}$.

  Let $L$ be a $4$-multicomplex seen as a left-$\Aa_4$-module, where $\Aa_4$ is the $R$-algebra introduced in \ref{algebrasnotation}. The left $\Aa_4$-module $q'(L)=L/\Aa_4\delta_0 L+\Aa_4\delta_3L$,  satisfies
      \[\delta_0\cdot x=\delta_3\cdot x=0, \text{ for all } x\in q'(L).\]
      In particular, by setting $q(L)^{p,q}=q'(L)^{2p-q,p}$, the bigraded module $q(L)$ is a bicomplex with $d_0$ and $d_1$ induced by the action on $q'(L)$ of $\delta_1$ and $\delta_2$ respectively. 
      It is clear that $q$ induces a functor $q:\cpx{4}\to\cpx{2}$.

\begin{prop}
  The pair of functors $(q,j)$ is a Quillen adjunction.  
  \[
    \begin{tikzcd}
      (\cpx{2})_{0,0} \arrow[r, "j"', "\perp" , shift right=1.5] &(\cpx{4})_{1,1} \arrow[l, "q"',shift right=1.5] 
    \end{tikzcd}
      \]  
      Moreover the functors $q$ and $j$ preserve the involution.
\end{prop}

\begin{proof}
    That the pair $(q,j)$ is an adjunction is immediate. The counit of the adjunction is the identity. The unit of the adjunction $L\to jqL$, for $L\in \cpx{4}$ is given by the projection $L\to L/\Aa_4\delta_0 L+\Aa_4\delta_3L$. Let us show that the right adjoint $j$ preserves fibrations and weak equivalences. Let $f:M\to N$ be a morphism in $\cpx{2}$ such that $\EI_0(f)=\EII_0(f)$ is surjective. Since the differentials on $\EI_0(j(f))$ and $\EII_0(j(f))$ are $0$ we get that $\EI_r(j(f))$ and $\EII_s(j(f))$ are surjective for $0\leq r,s\leq 1$. For the same reason we have that if $f\in\Ee_{0,0}$, then $j(f)\in\Ee_{1,1}$.
    That the functors $q$ and $j$ preserve the involution is immediate.
\end{proof}

\begin{rem}
    It is not clear to us whether $(q,j)$ is a Quillen equivalence. Since $j$ creates weak equivalences, $(q, j)$ is a Quillen equivalence if and only if
    the unit $\eta_L\colon L\to jqL$ is in $\Ee_{1,1}$ for every cofibrant object $L$ of $(\cpx{4})_{1,1}$.    

    The following provides an example of an object \(K\) such that $\eta_K$ is not in $\Ee_{1,1}$. Let
$K=\kk x\oplus \kk d_0x$ with $x$ in bidegree $(0,0)$: the first page of the spectral sequence associated to $K$ is $0$ whereas the spectral sequence associated to $jq(K)=\kk x$ never vanishes. However it can be checked that this \(K\) is not cofibrant, so this does not settle the question. 

We have checked that if $L$ is of the form $\Zzw_s^4(p,q)$ then $\eta_L$ is in $\Ee_{1,1}$. Despite the fact that the model category structure  $(\cpx{4})_{1,1}$ is cofibrantly generated, we have not been able to proceed to general cofibrant objects since the functor $jq$ does not preserve cofibrations.

We also note that we cannot use the standard argument saying that if $A\to B$ is a quasi-isomorphism of differential graded $\kk$-algebras then the categories of differential graded $A$-modules and $B$-modules are Quillen equivalent, since in that case, the weak equivalences are created by those in differential graded $\kk$-modules. But $\Ee_{1,1}$ is not a class of weak equivalences obtained in that way. 
    \end{rem}

\subsection{Geometric viewpoint}
The interest in comparing bicomplexes and 4-multicomplexes is the following: as mentioned in the introduction, 
the complexified de Rham algebra of any almost complex manifold $M$ splits as a direct sum 
\[\Aa^*_{\mathrm{dR}}(M)\otimes\mathbb{C}=\bigoplus_{p+q=*}\Aa^{p,q}.\]
This is
induced by the almost complex structure $J:TM\to TM$. In general, the differential decomposes into four components $d=\overline{\mu}+\overline{\partial}+\partial+\mu$ of different bidegrees, making the complexified de Rham algebra into a 4-multicomplex. 
A long-standing question in complex geometry is whether any compact almost complex manifold of dimension $\geq 6$ admits an integrable structure (and hence has a holomorphic atlas). Note that in dimension 2 this is always the case, while in dimension 4 there are known examples of almost complex manifolds not admitting any integrable structure. However, in higher dimensions the question remains open. In particular, it is not known whether $S^6$ admits the structure of a complex manifold. Integrable structures are characterized by the vanishing of $\mu$ and $\overline{\mu}$ and so
in this case, the complex de Rham algebra is just a bicomplex, with $d=\overline{\partial}+\partial$. 
We refer to Chapter 1 of \cite{Angella} or Chapter 4 of \cite{FOT} for proof and discussion of these facts.

A more refined question, posed by Sullivan, is whether there exist homotopical obstructions for a compact almost complex manifold to admit an integrable structure.
From the viewpoint of rational homotopy, this question should be addressed by comparing the homotopy theories of commutative algebras in 4-multicomplexes and commutative algebras in 2-multicomplexes, i.e., bidifferential bigraded commutative algebras. 
We note here that monoidal properties of the model structures of this paper have not been established.
One potential direction is to draw upon the work of Brotherston~\cite{Bro}, which establishes monoidal structures within the context of related model structures on filtered complexes.
With sufficiently well-behaved monoidal model structures in place, results
of White~\cite{White2017} would allow one to obtain model structures on commutative algebras.

\bibliographystyle{alpha}
\bibliography{biblio}
\end{document}